\documentclass[letterpaper, 10 pt, conference]{ieeeconf}  

\usepackage{amsmath}
\usepackage{amssymb}
\usepackage{hyperref}
\usepackage{bm}
\usepackage{xcolor}
\usepackage{graphicx}

\providecommand{\qed}{}

\newtheorem{theorem}{Theorem}
\newtheorem{lemma}{Lemma}
\newtheorem{proposition}{Proposition}

\newtheorem{assumption}{Assumption}

\newcommand{\assumptionref}[1]{Assumption~\ref{#1}}

\newcommand{\modif}[1]{\textcolor{black}{#1}}

\providecommand{\quotes[1]}{``#1"}
\newcommand{\ie}{i.e.}
\newcommand{\eg}{e.g.}
\newcommand{\R}{\mathbb{R}}
\newcommand{\T}{\top}
\newcommand{\defeq}{:=}
\newcommand{\ones}{\bm{1}}
\newcommand{\zeros}{\bm{0}}

\newcommand{\norm}[1]{\left\| #1 \right\|}
\newcommand{\order}[1]{\mathcal{O}\left( #1 \right)}
\newcommand{\minimize}{\mathop{\operatorname{minimize}}}

\providecommand{\prt}[1]{\left( #1 \right)}

\newcommand{\bx}{\bm{x}}
\newcommand{\by}{\bm{y}}
\newcommand{\bL}{\bm{L}}
\newcommand{\xbar}{\bar{x}}
\newcommand{\bxbar}{\bar{\bm{x}}}
\newcommand{\fa}{f_{\text{a}}}
\newcommand{\fb}{f_{\text{b}}}
\newcommand{\xa}{x^{\text{a}}}

\newcommand{\xb}{x^{\text{b}}}
\newcommand{\calF}{\mathcal{F}}

\newcommand{\fbar}{\bar{f}}
\newcommand{\alphabar}{\bar{\alpha}}
\newcommand{\betabar}{\bar{\beta}}
\newcommand{\kappabar}{\bar{\kappa}}

\IEEEoverridecommandlockouts                              
\title{\LARGE \bf
Stability of Decentralized Gradient Descent\\ in Open Multi-Agent Systems}

\author{Julien M. Hendrickx and Michael~G.~Rabbat
\thanks{
J.H. is with ICTEAM institute, UCLouvain (Belgium) {\tt\small julien.hendrickx@uclouvain.be}. 
His work was was supported by the \quotes{RevealFlight} Concerted Research Action (ARC) of the Federation Wallonie-Bruxelles, by the Incentive Grant for Scientific Research (MIS) \quotes{Learning from Pairwise Comparisons} of the F.R.S.-FNRS, and by a WBI World Excellence Fellowship.
M.R. is with Facebook AI Research, Montreal, Canada, 
        {\tt\small mikerabbat@fb.com}.%
}
}

\begin{document}

\maketitle
\thispagestyle{empty}
\pagestyle{empty}


\begin{abstract}
The aim of decentralized gradient descent (DGD) is to minimize a sum of $n$ functions held by interconnected agents. We study the stability of DGD in open contexts where agents can join or leave the system, resulting each time in the addition or the removal of their function from the global objective. Assuming all functions are smooth, strongly convex, and their minimizers all lie in a given ball, we characterize the sensitivity of the global minimizer of the sum of these functions to the removal or addition of a new function and provide bounds in 
$ O\prt{\min\prt{\kappa^{0.5}, \kappa/n^{0.5},\kappa^{1.5}/n}}$
where $\kappa$ is the condition number. We also show that the states of all agents can be eventually bounded independently of the sequence of arrivals and departures. The magnitude of the bound scales with the importance of the interconnection, which also determines the accuracy of the final solution in the absence of arrival and departure, exposing thus a potential trade-off between accuracy and sensitivity. Our analysis relies on the formulation of DGD as gradient descent on an auxiliary function.
The tightness of our results is analyzed using the PESTO Toolbox.
\end{abstract}

\section{Introduction}
\label{sec:intro}

Multi-agent consensus optimization methods aim to solve problems of the form
\begin{align} \label{eq:prob}
\minimize_{x_1, \dots, x_n \in \R^d} &\quad \sum_{i=1}^n f_i(x_i) \\
\text{subject to} &\quad x_i = x_j \text{ for all } i, j \text{ such that } a_{i,j} > 0.\nonumber
\end{align}
Here there are $n$ agents, agent $i \in [n]$ has access to a local objective function $f_i : \R^d \rightarrow \R$ and a local decision vector $x_i \in \R^d$, and the agents seek to reach agreement on the minimizer of $f(x) := \sum_{i \in [n]} f_i(x)$. The constraints couple the agents' local decision vectors through a communication graph with (weighted) adjacency matrix $A \in \R^{n \times n}_+$, where entry $a_{i,j}$ is (strictly) positive if and only if agent $i$ receives messages from agent $j$.

This class of methods has received much attention since the decentralized (sub)gradient descent method was introduced in~\cite{nedic2009distributed}. Previous work has primarily focused on extending these methods, e.g., by developing more efficient methods, understanding the dependence of the rate of convergence on $A$, considering directed and/or time-varying communication graphs, as well as various different assumptions about the class of functions being optimized; see~\cite{Nedic2018network} for a review.

Other works have also considered the case where the objective function may vary while the optimization method is being executed. Some studies formulate the problem as regret minimization~\cite{lee2017stochastic,shahrampour2018distributed} --- finding the single decision vector that minimizes the time-averaged objective function. Other studies aim to track the instantaneous minimizer in specific problem settings, such as linear regression~\cite{cattivelli2010diffusion}, or when the local objectives may vary smoothly in time~\cite{simonetto2017decentralized}. 

In this work we consider the setting in which one or more agent may change their objective function $f_i$ in a discontinuous manner as long as the objective remains in the class of smooth, strongly convex functions with bounded minimizer, and we characterize behavior with respect to the instantaneous minimizer. To motivate this setting, consider the following three scenarios.

1.~In open multi-agent systems~\cite{hendrickx2017open}, agents arrive and depart while the algorithm is executing. When an upper bound is available on the total number of agents that are active at any point in time, some such systems could be modeled by having agent objectives switch to $f_i(x) = 0$ when agent $i$ is inactive.

2.~When using clusters of servers to train large machine learning models, it is desirable to have \emph{elastic} algorithms where the number of servers may vary over time~\cite{narayanamurthy2013elastic}. Some servers may fail during execution and leave the system. In other cases, execution may start when some minimum number of servers is available, and additional servers may join during execution as they become available.

3.~In secure distributed learning~\cite{chen2018secure} and federated learning~\cite{mcmahan2017communication} systems, agents use techniques from secure multi-party computation to aggregate information without revealing the values of their individual gradients to any other agent in the network (including their neighbors). An outlier or adversarial agent may wish to exert influence by modifying their local objective $f_i$.

Our contributions are as follows. We characterize the extent to which a single agent can influence the minimizer when each function $f_i$ is smooth, strongly convex, and its minimizer has bounded norm. Then we go on to study the extent to which one agent can move the minimizer by changing its objective function. Finally, we establish that decentralized gradient descent is stable: there exists a bounded set such that once the iterates of the decentralized gradient method enter the set they remain in the set even if agents change their objective at every iteration. Hence, agents cannot drive the iterates to become unbounded.

\section{Problem formulation and notation}
\label{sec:prob}

Throughout this work, we denote the Euclidean norm by $\norm{\cdot}$; so $\norm{x} = (x^\T x)^{1/2}$ for a vector $x$, and $\norm{A}$ denotes the largest singular value of a matrix $A$. The set of the first $n$ natural numbers is denoted by $[n]$, and $\R_+$ denotes the set of non-negative real numbers.

\subsection{Objective functions}
The following assumptions about the local objectives $f_i$ hold throughout this paper.
\begin{assumption} \label{ass:f_i}
Each function $f_i$, $i \in [n]$, is continuously differentiable, $\alpha$-strongly convex (\modif{$f_i(x) - \frac{\alpha}{2}\norm{x}^2$ is convex}) and $\beta$-smooth (\modif{$\norm{\nabla f_i(x) - \nabla f_i(y)} \le \beta \norm{x - y}$, $\forall x,y$}).  
\end{assumption}

\modif{We denote by $\mathcal{F}_{\alpha, \beta}$ the set of functions satisfying assumption \assumptionref{ass:f_i}. The \emph{condition number} of these functions is $\kappa = \beta/\alpha$, and plays an important role in the convergence analysis of algorithms optimizing such functions.}

\modif{The strong convexity part of \assumptionref{ass:f_i}} implies that $f = \sum_{i \in [n]} f_i$ is strongly convex, and hence has a unique global minimizer which we denote by $x^* \in \R^d$. If a local function $f_i$ is allowed to change arbitrarily, then clearly the minimizer of $f$ can also be changed arbitrarily. To disallow this, we impose a form of normalization on the local objectives. Let $B(x,r) = \{y : \norm{x - y} \le r\}$ denote the ball of radius $r$ centered at $x$.

\begin{assumption} \label{ass:x_i*}
Let $x_i^* \defeq \arg\min_x f_i(x)$ denote the minimizer of $f_i$. For each $i \in [n]$, we assume that $f_i(x_i^*) = \min_x f_i(x) = 0$ and that $x_i^* \in B(\zeros_d, 1)$, where $\zeros_d$ denotes the vector of zeros in $\R^d$.
\end{assumption}

The minimizer of a strongly convex function is finite, so the assumption that $x_i^* \in B(0,1)$ for all $i$ can be seen as holding generally by rescaling the space. Similarly, since we will mainly be concerned with understanding how $x^*$ can change if one $f_i$ changes and we do not use the actual values of the $f_i$, there is no loss of generality by assuming that $f_i(x_i^*) = 0$. 
In the sequel, we denote by $\calF_{\alpha, \beta}$ the set of functions that simultaneously satisfy Assumptions~\ref{ass:f_i} and~\ref{ass:x_i*}.

\subsection{Communication graph and mixing matrix}

Recall that communication constraints, describing which agents communicate directly, are captured by an agent mixing matrix $A \in \R_+^{n \times n}$, where $a_{i,j} > 0$ if and only if agent $i$ receives messages from agent $j$. The following assumptions about $A$ hold throughout.
\begin{assumption} \label{ass:A}$ $
\begin{enumerate}
\renewcommand{\theenumi}{\roman{enumi}}
\item \modif{The matrix $A$ is symmetric.}
\item All diagonal entries $a_{i,i} > 0$ are positive.
\item The matrix $A$ is primitive: there exists an integer $k > 0$ such that every entry of $A^k$ is positive.
\end{enumerate}
\end{assumption}

It is common to view $A$ as the (weighted) adjacency matrix of a graph $G = (V, E)$, with vertices $V = [n]$ corresponding to agents and edges $E = \{(i,j) : a_{i,j} > 0, i \ne j\}$ between agents that exchange messages. In this case, the third item in \assumptionref{ass:A} is equivalent to assuming that the graph $G$ is connected.

We will make use of standard concepts from spectral graph theory. Let $D$ denote a $n \times n$ diagonal matrix with entries $D_{i,i} = \sum_j a_{i,j}$ equal to the weighted degree of agent $i$. The graph Laplacian matrix is $L = D - A$. 
Since $A$ is symmetric, by \assumptionref{ass:A}, $L$ is too, and so it has an eigendecomposition. Let $\lambda_1 \le \lambda_2 \le \dots \le \lambda_n$ denote the $n$ eigenvalues of $L$ sorted in ascending order. It follows from \assumptionref{ass:A}(iii) that $\lambda_1 = 0$ and $\lambda_2 > 0$, therefore $L$ is positive semi-definite. In addition, the null space of $L$ is spanned by $\ones_d$, the $d$-dimesional vector of ones.

Let $I_d$ denote the $d \times d$ identity matrix, and let $\bL = L \otimes I_d$. Let $\bx \in \R^{nd}$ denote the vector obtained by stacking $x_1, \dots, x_n$, where $x_i \in \R^d$ is the decision vector at agent $i$. One can verify that
\begin{equation}\label{eq:bxbLbx}
\bx^\T \bL \bx = \sum_{i \in [n]} \sum_{j > i} a_{i,j} \norm{x_i - x_j}^2.
\end{equation}
Similarly, if we let $\by = \bL \bx \in \R^{nd}$ denote a vector with blocks $y_1, \dots, y_n \in \R^d$, then one can verify that
\begin{equation}\label{eq:bLbx_i}
y_i = - \sum_{j=1}^n a_{i,j} (x_i - x_j).
\end{equation}

\subsection{Decentralized Gradient Descent (DGD)}\label{sec:def_DGD}

DGD is commonly defined by the iterations
\begin{equation} \label{eq:dgd}
x_i^{k+1} = \sum_{j=1}^n \tilde a_{i,j} x_i^k - \eta \nabla f_i(x_i^k),
\end{equation}
with stochastic matrix $\tilde A$~\cite{nedic2009distributed}. Here we consider a more general form of DGD given by
\begin{align} \label{eq:algo}
x_i^{k+1} &= x_i^k - \eta\left(\nabla f_i(x_i^k) + \rho \sum_{j=1}^n a_{i,j} \modif{(x_i^k - x_j^k)} \right).
\end{align}
Note that \eqref{eq:algo} and \eqref{eq:dgd} are equivalent provided that $\eta \rho$ is not too large, as can be seen by taking $\tilde a_{i,j} = \eta \rho a_{i,j}$ for $i\neq j$ and $\tilde a_{i,i} = 1 - \eta\rho \sum_j a_{i,j}$.

We study \eqref{eq:algo} because it can be seen as applying gradient descent to a penalized objective function \cite{yuan2016convergence}. Recall that $\bx \in \R^{nd}$ denotes the vector obtained by stacking $x_1, \dots, x_n$, and define $F(\bx) \modif{\defeq} \sum_{i \in [n]} f_i(x_i)$. One way to incorporate the constraints in \eqref{eq:prob} is though quadratic penalties: 
\begin{align} \label{eq:def_Frho}
F_\rho(\bx) & \defeq F(\bx) + \frac{\rho}{2} \sum_{i \in [n]} \sum_{j > i} a_{i,j} \norm{x_i - x_j}^2 \\
&= F(\bx) + \frac{\rho}{2} \bx^T \bL\bx.\nonumber
\end{align}
where $\rho > 0$ controls the weight given to the penalty relative to the objective 
$F$ 
and we have used \eqref{eq:bxbLbx} for the second \modif{equality.} Minimizing $F_\rho$ leads to a trade-off governed by $\rho$ between the equality of the $x_i$ and the minimization of the $f_i$. 

Applying the gradient method with constant step-size $\eta > 0$ to minimize $F_\rho$ leads to updates of the form
$$
\bx^{k+1} = \bx^k- \eta (\nabla F(\bx) + \rho \bL \bx),
$$
which is equivalent to \eqref{eq:algo}. By observing that the updates~\eqref{eq:algo} correspond to the gradient method applied to $F_\rho$, standard  results (\eg,~\cite{nesterov2003intro,bubeck2015convex}) directly give that the iterates \eqref{eq:algo} converge linearly to the minimizer of $F_\rho$ when the step size $\eta$ is chosen appropriately.

The following result establishes that $f$, $F$, and $F_\rho$ are smooth and strongly convex; the proof involves standard algebraic manipulations and \modif{is omitted for space reasons.}

\begin{proposition} \label{prop:f_properties}
Suppose that $f_1, \dots, f_n \in \calF_{\alpha, \beta}$. Then:
\renewcommand{\theenumi}{\roman{enumi}}
\begin{enumerate}
\item $F$ is $\alpha$-strongly convex and $\beta$-smooth.
\item $F_\rho$ is $\alpha$-strongly convex and $(\beta + \rho \lambda_n)$-smooth, where $\lambda_n$ is the largest eigenvalue of $L$.
\item $f$ is $n\alpha$-strongly convex and $n\beta$-smooth.
\end{enumerate}
\end{proposition}

\section{Location of the minimizers}

Since $f$, $F$, and $F_\rho$ are strongly convex, they each have a unique minimizer. Moreover, since $F(\bx)$ is separable in $x_1, \dots, x_n$, its minimizer $\bx^* = \arg\min_{\bx} F(\bx)$ is trivially the vector obtained by stacking $x_1^*, \dots, x_n^*$, where $x_i^* = \arg\min_x f_i(x)$, and thus $\bx^* \in B(\zeros_d, 1)^n$.

\begin{figure}
\centering
\begin{tabular}{cc}
\includegraphics[scale =.35]{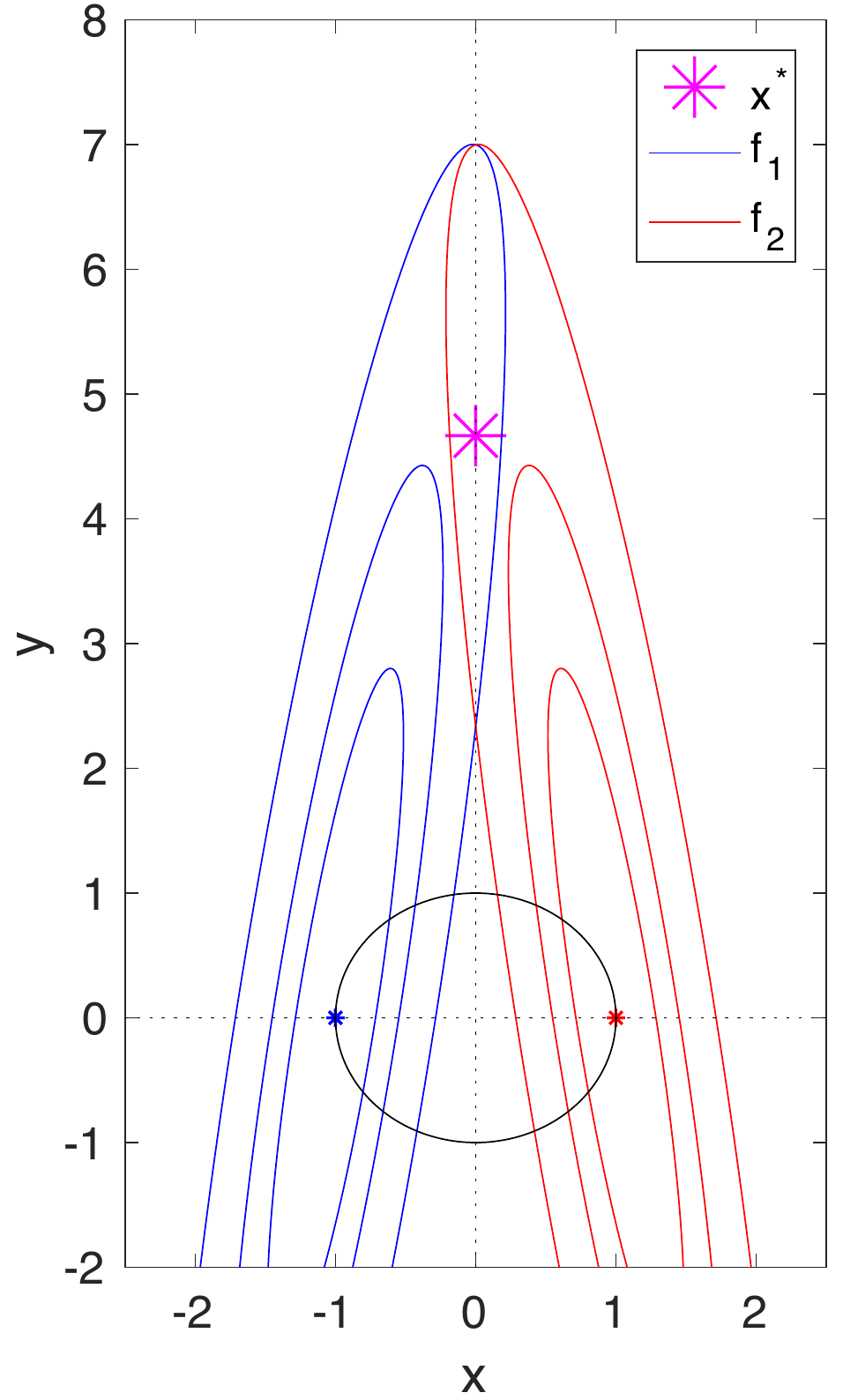}&\includegraphics[scale =.35]{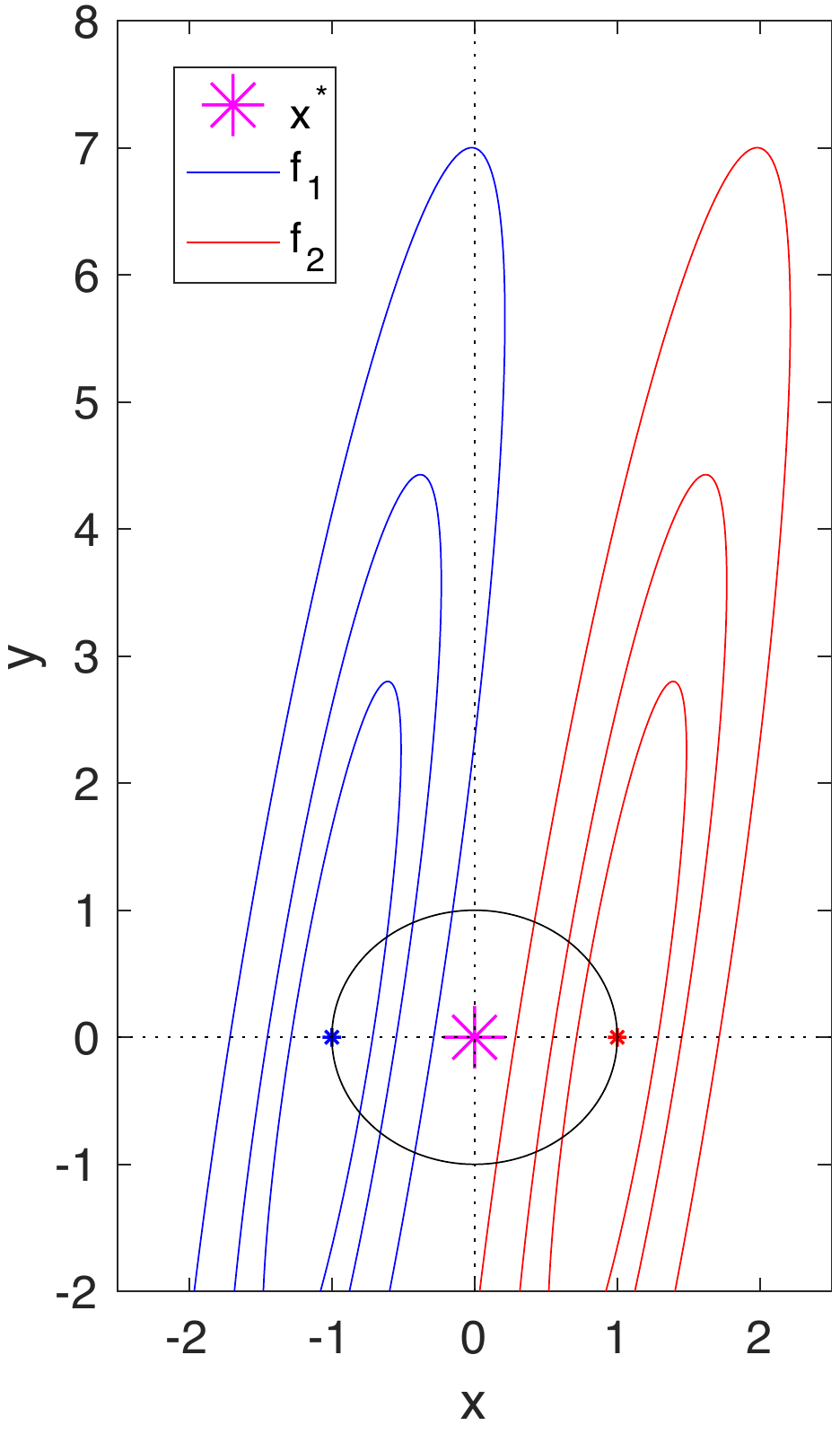}\\(a)&(b)
\end{tabular}
\caption{Representation of the level sets of $f_1$ and $f_2$ defined in \eqref{eq:ex_function_1} in (a), and of $f_1$ from \eqref{eq:ex_function_1} and $f_b$ form \eqref{eq:ex_fb} (b), together with the minimizer $x^*$ of their sum, for $\kappa = 100$ and $ \phi = \tan^{-1}\kappa^{-1/2}\simeq 0.0997$ (rad) . One can see the minimizer can be at a distance $\simeq \sqrt{\kappa}/2$ of 0, and that a small rotation of a function can have an important effect on its position.
}\label{fig:ex_functions}
\end{figure}

\modif{Since the minimizers $x_i^*$ of the $f_i$ are all in $B(0_d,1)$ in view of Assumption \ref{ass:x_i*},}
one may have hoped that $x^*=\arg\min f(x)=\sum_{i\in [n]}f_i(x)$ would be close to that ball. Unfortunately, this is far from true, as \modif{shown} by the following example. Consider the following two quadratic functions 
\begin{align}\label{eq:ex_function_1}
    f_1(x,y) &= \frac{\beta}{2}((x + 1)\cos \phi-y \sin \phi )^2 \nonumber \\
    &\quad+ \frac{\alpha}{2}( (x + 1) \sin \phi+y \cos \phi)^2 \\
    f_2(x,y) &= \frac{\beta}{2}((x - 1)\cos \phi+y \sin \phi )^2 \nonumber \\
    &\quad + \frac{\alpha}{2}( - (x + 1) \sin \phi+y \cos \phi)^2 
\end{align}
for scalar coordinates $x$ and $y$ and some angle $\phi$, as illustrated in Figure \ref{fig:ex_functions}. Their minimizers are clearly $(-1,0)$ and $(1,0)$, both in $B(0_2,1)$, and one can verify that they are $\alpha$-strongly convex and $\beta$-smooth, as they can be obtained by applying a rotation $\pm\phi$ to $\frac{\modif{\alpha}}{2}x^2 + \frac{\modif{\beta}}{2}y^2$ followed by a translation.
Simple derivations show that the minimum of $f_1+f_2$ is reached at $\prt{0,\frac{(\kappa-1)\tan \phi}{1+\kappa \tan^2\phi }}$, which, for $\tan \phi =1/\sqrt{ \kappa}$ becomes
$\prt{0,\frac{\kappa - 1}{2\sqrt{\kappa}}}$. So even though the minimizer of each function is in $B(0_2,1)$, the norm of the minimizer of their sum scales as $\order{\sqrt{\kappa}}$. The next result shows that this is indeed the worst scaling possible up to a constant factor.

\begin{theorem} \label{thm:x*localization}
Let $\kappa = \beta / \alpha$ denote the condition number of $f$, and let $R_\kappa \defeq 1 + \sqrt{\kappa}$. If $f_i \in \calF_{\alpha, \beta}$ for all $i=1,\dots,n$, then the following two inclusions, \modif{respectively in $\R^d$ and $\R^{nd}$}, hold:
\begin{align}
&\arg\min_{x} f(x) \in B(\zeros_d, R_\kappa), \text{ and } \label{eq:x*location} \\
&\arg\min_{\bx} F_\rho(\bx) \in B(\zeros_d, R_\kappa)^n. \label{eq:bx*location}
\end{align}
\end{theorem}

\begin{proof}
By \assumptionref{ass:x_i*}, we have $f_i(x_i^*) = 0$ and $x_i^* \in B(0,1)$. Since $f_i$ is $\beta$-smooth, this implies that
\begin{equation}
f_i(\zeros_d) \le \frac{\beta}{2}. \label{eq:f_i0upperbound}
\end{equation}
For any $x \notin B(\zeros_d, R_\kappa)$ we have $\norm{x - x_i^*} > R_\kappa - 1 = \sqrt{\kappa}$ since $\norm{x} > R_\kappa$. Then $\alpha$-strong convexity of $f_i$ implies that
\begin{equation}
f_i(x) > \frac{\alpha \kappa}{2} = \frac{\beta}{2}, \quad \text{ for all } x \notin B(\zeros_d, R_\kappa). \label{eq:f_ixlowerbound}
\end{equation}
Using this bound along with~\eqref{eq:f_i0upperbound}, we obtain that for any $x \notin B(\zeros_d, R_\kappa)$,
\begin{equation} \label{eq:f_ixnotinball}
f_i(x) > f_i(\zeros_d),
\end{equation}
and thus
\[
f(x) > \frac{n\beta}{2} \ge f(\zeros_d) \ge \min_{\xbar \in \R^d} f(\xbar).
\]
Therefore, $\arg\min_x f(x) \in B(\zeros_d, R_\kappa)$, which completes the proof of \eqref{eq:x*location}.

Next, consider a point $\bx = (x_1, x_2, \dots, x_n) \notin B(\zeros_d, R_\kappa)^n$, and let $\bx' = (x'_1, \dots, x'_n)$ denote the projection of $\bx$ onto $B(\zeros_d, R_\kappa)^n$; i.e.,
\[
x'_i \defeq \begin{cases} x_i & \text{ if } x_i \in B(\zeros_d, R_\kappa), \\ R_\kappa \frac{x_i}{\norm{x_i}} & \text{ otherwise.} \end{cases}
\]
We will show that $F_\rho(\bx)$ is larger than $F_\rho(\bx')$, which then implies the claim~\eqref{eq:bx*location}. Since a projection onto a convex set is non-expansive, we have $\norm{x'_i - x'_j} \le \norm{x_i - x_j}$ for all $i,j \in [n]$, and thus $(\bx')^\T \bL \bx' \le \bx^\T \bL \bx$. Moreover, there is at least one $i \in [n]$ such that $x'_i \ne x_i$ because
$\bx \notin B(\zeros_d, R_\kappa)^n$, and for every such $i$ we know that $x'_i$ is a convex combination of $x_i$ and $\zeros_d$; \ie, $x'_i = \theta x_i + (1 - \theta) \zeros_d$ for some $\theta \in (0,1)$. For those indices $i$ such that $x'_i \ne x_i$, we have $x_i \notin B(\zeros_d, R_\kappa)$, and so using convexity of $f_i$ and \eqref{eq:f_ixnotinball}, we get
\[
f_i(x'_i) \le \theta f_i(x_i) + (1 - \theta) f_i(\zeros_d) < f_i(x_i).
\]
Thus, for any $\bx \notin B(\zeros_d, R_\kappa)^n$, we have $F(\bx) > F(\bx')$, and therefore $F_\rho(\bx) > F_\rho(\bx') \ge \min_{\bxbar} F_\rho(\bxbar)$. Hence, $\arg\min_{\bx} F_\rho(\bx) \in B(\zeros_d, R_\kappa)^n$, \modif{which establishes \eqref{eq:bx*location}.} \qed
\end{proof}
\smallskip
The bound for $f$ is tight up to a constant factor even for $n=2$, as shown by the example \eqref{eq:ex_function_1}. The bound for $F_\rho$ should be tight for large $\rho$ as its minimizer approaches $1_n \otimes x^*$. For small $\rho$, $F_\rho$ is close to $F$ and its minimizer should be closer to $B(0_d,1)^n$. However, large values of $\rho$ is the relevant regime, as the goal of the minimization of $F_\rho$ is to find approximations of  $x^*$.

\section{Impact of Function change}\label{sec:function_change}

We now study to what extent the minimizer $x^* = \arg\min_x f(x)$ can change when the local objective $f_i$ at one agent changes. Considering a single change is justified if changes are sufficiently infrequent to only have to consider one at the time. Suppose we have $n+1$ functions $f_1, f_2, \dots, f_{n-1}, \fa, \fb \in \calF_{\alpha, \beta}$. Let $\xa$ and $\xb$ be defined as the minimizers,
\begin{align}
\xa &\defeq \arg\min_x \big( f_1(x) + f_2(x) + \dots + f_{n-1}(x) + \fa(x) \big) \nonumber\\
\xb &\defeq \arg\min_x \big( f_1(x) + f_2(x) + \dots + f_{n-1}(x) + \fb(x) \big). \label{eq:xab}
\end{align}
We wish to analyze $\norm{\xa - \xb}$. When $n=2$, $\xa$ and $\xb$ can be very different even if $\fa$ and $\fb$ are very similar. If $\fa = f_2$ as defined in \eqref{eq:ex_function_1}, then we have seen that the minimizer of $f_1 + \fa$ is $\prt{0,\frac{\kappa - 1}{2\sqrt{\kappa}}}$. Consider now \begin{align}\label{eq:ex_fb}
    \fb(x,y) &= 
    \frac{\beta}{2}((x+1) \cos \phi - y \sin \phi  )^2
    \\\nonumber &+
    \frac{\alpha}{2}((x + 1) \sin \phi + y \sin \phi  )^2 
    ,
\end{align}
with $\tan \phi = \sqrt{\kappa}$. Note $f_b$ can be obtained by applying a rotation of $2\phi= 2/\sqrt{\kappa}$ to $\fa$ around its minimizer, which is small if the condition number is large, see Figure \ref{fig:ex_functions}. Nevertheless, the minimum of $f_1+f_b$ is $(0,0)$, so changing $\fa$ to $\fb$ moves the minimizer by $\Omega(\sqrt\kappa)$.

A first bound on $\norm{\xa - \xb}$ in $O(\sqrt{\kappa})$ follows from Theorem \ref{thm:x*localization} as both $\xa$ and $\xb$ belong to $B(\zeros_d, R_\kappa)$. The next theorem presents stronger bounds exploiting $n$, as modifying a function clearly has a smaller potential impact when $n$ is large. Its proof is presented in the Appendix. 

\begin{theorem}\label{thm:bound_xa-xb}
For $\xa$ and $\xb$ defined in \eqref{eq:xab},
\[
\norm{\xa - \xb} \le \min\left( 4 + 4 \sqrt{\kappa}, \frac{4 \sqrt{\kappa} + 2\kappa}{\sqrt{n-1}}, \frac{4\kappa + 2\kappa^{3/2}}{n} \right).
\]
\end{theorem}

The tightness of Theorem~\ref{thm:bound_xa-xb} can be analyzed empirically using the PESTO toolbox \cite{taylor2017performance}, initially developed to compute the exact worst-case performances of optimization algorithms. This toolbox allows indeed solving exactly (up to numerical precision) optimization problems of the form $\max \norm{\xa - \xb}$ over all functions $f_1, f_2, \dots, f_{n-1}, \fa, \fb \in \calF_{\alpha, \beta}$ satisfying our assumptions, and $\xa,\xb$ defined as in \eqref{eq:xab}. It allows thus computing the worst possible value for $\norm{\xa - \xb}$ appearing in Theorem \ref{thm:bound_xa-xb}, for given $n$ and $\kappa$. Figure \ref{fig:evol_kappa} clearly indicates that this worst-case value evolves as $\sqrt{\kappa}$ for a given $n$. Similar tests show a decrease of the worst-case distance with $n$ for a given $\kappa$, but no clear expression of this decay was identified so far. This suggests the tightness of our bound for fixed $n$ and growing $\kappa$, but possible improvements in other regimes. 

\begin{figure}
    \centering
    \includegraphics[scale=.38]{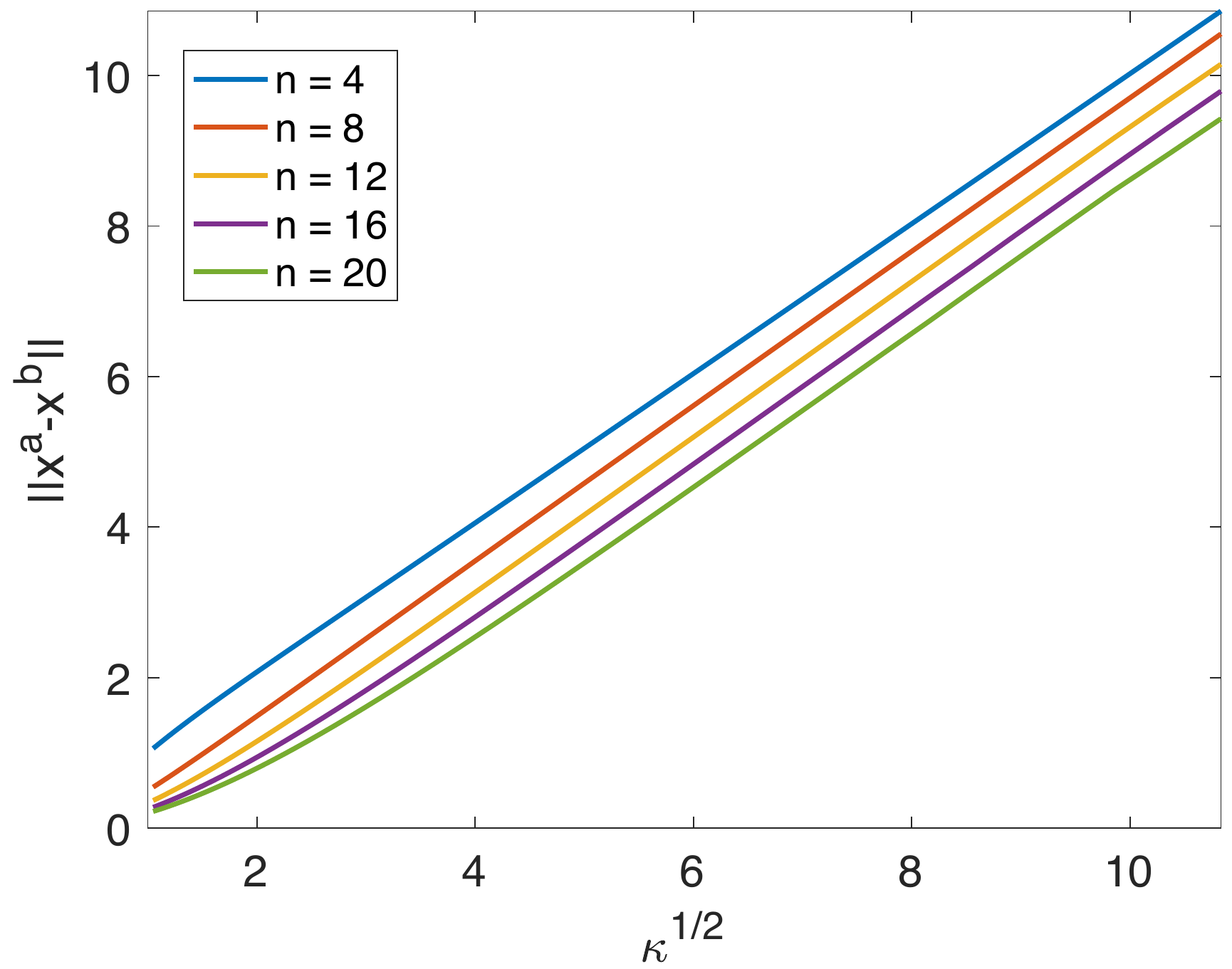}
    \caption{Evolution with $\sqrt{\kappa}$ of the \modif{exact worst-case bound on the distance} 
    $\norm{\xa-\xb}$ for $\xa,\xb$ defined in \eqref{eq:xab}. \modif{These bounds are computed numerically using the PESTO toolbox for performance estimation \cite{taylor2017performance}. Results suggest a clear asymptotic scaling in $\sqrt{\kappa}$ for fixed $n$.}
    }
    \label{fig:evol_kappa}
    
\end{figure}

\section{Stability of DGD}

We have seen in Section \ref{sec:function_change} that modifications of certain $f_i$ may result in large variations of the minimizer. If these happen during the execution of decentralized gradient descent (DGD), they can lead to sudden dramatic increases of the distance to the optimum. \modif{We now study} the stability of the DGD iteration \eqref{eq:algo} in this context, taking advantage of its formulation as a gradient descent on the objective $F_\rho(\bx)$. 

Recall the following standard result on gradient descent; see, \eg,~\cite{nesterov2003intro,bubeck2015convex}.

\begin{theorem} \label{thm:gradient-method-convergence}
Let $\fbar : \R^d \rightarrow R$ be a $\betabar$-smooth, $\alphabar$-strongly convex function minimized at $x^* = \arg\min_x \fbar(x)$. Let $\eta$ be a positive scalar which is at most $1/\betabar$. For any initial point $x^0 \in \R^d$, the sequence $\{x^k$, $k > 0\}$, produced by the gradient method $x^{k+1} = x^k - \eta \nabla \fbar(x^k)$ satisfies
\begin{align}
\norm{x^{k+1} - x^*}^2 &\le (1 - \eta \alphabar) \norm{x^{k} - x^*}^2 \label{eq:gradient-method-step}\\
&\le (1 - \eta \alphabar)^{k+1} \norm{x^0 - x^*}^2 . \nonumber
\end{align}
\end{theorem}

In addition to establishing linear convergence of the gradient method, this standard result ensures monotonic progress is made towards the minimizer at every step. Recall that applying the gradient method to the objective $F_\rho$ in \eqref{eq:def_Frho} leads to the decentralized algorithm \eqref{eq:algo}. Thus, the iterates $\bx^k$ \modif{of} \eqref{eq:algo} converge monotonically to $\bx^\rho$, the minimizer of $F_\rho$, provided $\eta < 1/(\beta + \rho\lambda_n)$ in view of Proposition \ref{prop:f_properties}.

If the local objective function $f_i$ at one or more agent may change during the execution of the optimization method, clearly convergence is no longer guaranteed, but we will show and quantify the stability of DGD under arbitrary changes of functions satisfying our assumptions. 
Specifically, we study the dynamics 
\begin{align} \label{eq:dynamic-algo}
x_i^{k+1} &= x_i^k - \eta\left(\nabla f^k_i(x_i^k) + \rho \sum_{j=1}^n a_{i,j} (x_i - x_j) \right),
\end{align}
where the $f_i^k$ can thus change at each iteration $k$. This can be seen as 
\modif{a gradient descent} on the time-varying objective
\begin{equation} \label{eq:dynamic-objective}
F^k_\rho(\bx) = F^k(\bx) + \frac{\rho}{2} \bx^\T \bL \bx,
\end{equation}
with $F^k(\bx) = \sum_{i=1}^n f_i^k(x_i)$. When $f_i^k \in \mathcal{F}_{\alpha, \beta}$ for all $i \in [n]$ and $k \ge 0$, we know from Theorem~\ref{thm:x*localization} that, for all $k \ge 0$, the minimizer of $F^k_\rho$ lies in the ball $B(\zeros_d, R_\kappa)^n$, where $R_\kappa = 1 + \sqrt{\kappa}$. We will \modif{show} that, even when every $f_i^k$ may vary from step to step, the dynamics \eqref{eq:dynamic-algo} cannot cause the sequence $\bx^k$ to grow unbounded if $f_i^k \in \mathcal{F}_{\alpha, \beta}$ for all $k \ge 0$. 

\begin{theorem} \label{thm:stability}
Consider the sequence $\bx^k$ generated by updates \eqref{eq:dynamic-algo} with possibly time-varying objective $F^k_\rho$ (equiv., $f_i^k$). Assume that $f_i^k \in \mathcal{F}_{\alpha, \beta}$ for all $i \in [n]$ and $k \ge 0$, and $\eta \leq 1 / (\beta + \rho \lambda_n)$.
Let $\kappa_\rho = \frac{\beta + \rho \lambda_n}{\alpha}$ and let $\kappa = \beta / \alpha$. There exists a finite positive constant $R$,
\begin{equation*}
R = \begin{cases} (1+\sqrt{6}) \modif{\sqrt{n}}(1 + \sqrt{\kappa}) & \text{ if } \kappa_\rho < 3 \\
(1 + \sqrt{2 \kappa_\rho}) \sqrt{n} (1 + \sqrt{\kappa}) & \text{ otherwise,} 
\end{cases}
\end{equation*}
such that if $\norm{\bx^{k_0}} \le R$ at some iteration $k_0$, then $\norm{\bx^k} \le R$ for all $k \ge k_0$.
\end{theorem}

A minor extension of our result not presented here shows the iterates $\bx^{k}$ approach $B(0,R)$ if $\norm{\bx^{k_0}}\geq R$. Moreover, since the proof is based on the analysis of a single iteration, Theorem \ref{thm:stability} can directly be extended to time-varying network matrices with uniformly bounded $\lambda_n$. Finally, observe that the bound $R$ scales with the square root of the condition number $\kappa_\rho$ of $F_\rho$, which itself grows linearly with  the penalty parameter $\rho$. This quantity can also be shown to determine the accuracy of the points to which each agent $i$ converges. Hence our result reveals a potential trade-off between accuracy and stability.

\subsection{Intermediate Results on Gradient Descent}

The proof of Theorem~\ref{thm:stability} builds on an intermediate result in a more general setting: Let $\fbar: \R^d \rightarrow \R$ be an $\alphabar$-strongly convex and $\betabar$-smooth function, and let $\kappabar = \betabar / \alphabar$. Let $x^* = \arg\min_x \fbar(x)$ denote the minimizer of $\fbar$, and suppose that there exists a finite positive scalar $b$ such that $\norm{x^*} \le b$. Let $\eta$ be a non-negative scalar that is at most $1/\betabar$. For any initial point $x \in \R^d$, let $x^+ = x - \eta \nabla \fbar(x)$. We are interested in finding a constant $R$ such that if $\norm{x} \le R$ then $\norm{x^+} \le R$, i.e. the radius of a ball centered on 0 stable under an iteration of gradient descent on $\bar f$.

The next lemma allows us to simplify our focus to the case where $\eta = 1 / \betabar$.

\begin{lemma} \label{lem:eta=1/beta}
Let $R$ be a finite positive scalar. If $\norm{x} \le R$ and $\norm{x - (1/\betabar)\nabla \fbar(x)} \le R$, then $\norm{x - \eta \nabla \fbar(x)}\le R$ for any $\eta$ satisfying $0 \le \eta \le 1/\betabar$.
\end{lemma}

\begin{proof}
By hypothesis, 
$x$ and $x - (1/\betabar)\nabla \fbar(x)$ are 
in the convex set $B(\zeros_d, R)$. 
Now, 
for $0 < \eta < 1/\betabar$, the point $x - \eta \nabla \fbar(x)$ is a convex combination of $x$ and $x - (1/\betabar)\nabla \fbar(x)$, and hence also belongs to $B(\zeros_d, R)$ by convexity.
\qed
\end{proof}
\smallskip

The next proposition provides a stability result provided $x$ is large enough. 
\begin{proposition} \label{prop:stability-large-kappa}
Assume  $\norm{x} \ge b$ and $\kappabar \ge 3$, and let
\begin{equation} \label{eq:stability-radius}
R = (1 + \sqrt{2 \kappabar})b.
\end{equation}
If $\norm{x} \le R$, then $\norm{x^+} \le R$.
\end{proposition}

\begin{proof}
By standard conditions for smoothness and strong convexity of $\fbar$, see e.g.~\cite{nesterov2003intro}, we have for any $x$, 
\begin{align*}
\nabla \fbar(x)^\T (x - x^*) &\ge 
\frac{\betabar^{-1}\norm{\nabla \fbar(x)}^2 }{1 + \kappabar^{-1}} + \frac{\alphabar\norm{x - x^*}^2}{1 + \kappabar^{-1}} .
\end{align*}
Adding $\nabla \fbar(x)^\T x^*$ to both sides, applying the Cauchy-Schwartz inequality, and recalling that $\norm{x^*} \le b$ by assumption gives
\begin{align}
\nabla \fbar(x)^\T x &\ge \nabla \fbar(x)^\T x^* + \frac{\betabar^{-1}\norm{\nabla \fbar(x)}^2}{1 + \kappabar^{-1}}  + \frac{\alphabar \norm{x - x^*}^2}{1 + \kappabar^{-1}} \nonumber \\
&\ge - b \norm{\nabla \fbar(x)} + \frac{\betabar^{-1} \norm{\nabla \fbar(x)}^2 }{1 + \kappabar^{-1}}+ \frac{\alphabar\norm{x - x^*}^2}{1 + \kappabar^{-1}} . \label{eq:h-smooth-strongly-convex}
\end{align}
We focus on the case $\eta = 1 / \betabar$. By definition of $x^+$,
\begin{align*}
\norm{x^+}^2 &= \norm{x - \frac{1}{\betabar} \nabla \fbar(x)}^2 \\
&= \norm{x}^2 + \frac{1}{\betabar^2} \norm{\nabla \fbar(x)}^2 - 2 \frac{1}{\betabar} \nabla \fbar(x)^\T x.
\end{align*}
Using \eqref{eq:h-smooth-strongly-convex}, we can then bound $\norm{x^+}^2$ 
\begin{align}
&\norm{x^+}^2 \leq \norm{x}^2 + \frac{1}{\betabar^2} \norm{\nabla \fbar(x)}^2 \nonumber \\ 
&+\frac{2}{\betabar}\left(b \norm{\nabla \fbar(x)} - \frac{\betabar^{-1}\norm{\nabla \fbar(x)}^2}{1 + \kappabar^{-1}}  - \frac{\alphabar\norm{x - x^*}^2 }{1 + \kappabar^{-1}} \right) \nonumber \\
&= \norm{x}^2 + \frac{2}{\betabar} \left(b \norm{\nabla \fbar(x)} - \frac{1}{2\betabar} \left(\frac{1 - \kappabar^{-1}}{1 + \kappabar^{-1}}\right)\norm{\nabla \fbar(x)}^2 \right.\nonumber 
\\& \qquad \qquad \left.  - \frac{\alphabar}{1 + \kappabar^{-1}}\norm{x - x^*}^2 \right). \label{eq:intermediate-quadratic}
\end{align}
Since $\kappabar \ge 3$ by assumption, we have 
$\frac{1 - \kappabar^{-1}}{1 + \kappabar^{-1}} > 0 $.
Thus
\[
b \norm{\nabla \fbar(x)} - \frac{1}{2\betabar} \left(\frac{1 - \kappabar^{-1}}{1 + \kappabar^{-1}}\right)\norm{\nabla \fbar(x)}^2 - \frac{\alphabar\norm{x - x^*}^2}{1 + \kappabar^{-1}}
\]
is strictly concave in $\norm{\nabla \fbar(x)}$, and hence can be bounded by its maximum (w.r.t. to $\norm{\nabla \fbar(x)}$)
\[
\frac{b^2 \betabar}{2} \left(\frac{1 + \kappabar^{-1}}{1 - \kappabar^{-1}}\right) - \frac{\alphabar}{1 + \kappabar^{-1}}\norm{x - x^*}^2.
\]
Using this in \eqref{eq:intermediate-quadratic} gives that
\begin{align*}
\norm{x^+}^2 &\le \norm{x}^2 + b^2 \frac{1 + \kappabar^{-1}}{1 - \kappabar^{-1}} - 2 \frac{\kappabar^{-1}}{1 + \kappabar^{-1}}\norm{x - x^*}^2 \\
&\le \norm{x}^2 + b^2 \frac{1 + \kappabar^{-1}}{1 - \kappabar^{-1}} - 2 \frac{\kappabar^{-1}}{1 + \kappabar^{-1}}(\norm{x} - \norm{x^*})^2.
\end{align*}
For $\kappabar \ge 3$, we have $\frac{1 + \kappabar^{-1}}{1 - \kappabar^{-1}} \le 2$ and $-\frac{2\kappabar^{-1}}{1 + \kappabar^{-1}} \le - \kappabar^{-1}$. Furthermore, since $\norm{x} \ge b$, we have $\norm{x} - b \ge 0$. 
\begin{align*}
Q(\norm{x}) &= \norm{x}^2 + 2 b^2 - \kappa^{-1}(\norm{x} - b)^2.
\end{align*}
Now observe that $Q(\norm{x})$ is strictly convex in $\norm{x}$ since $1 - \kappabar^{-1} > 0$. Therefore, under the assumption that $0 \le \norm{x} \le R$, $Q(\norm{x})$ achieves its maximum when either $\norm{x} = 0$ or $\norm{x} = R$. Observe that 
\[
Q\left((1 + \sqrt{2 \kappa})b\right) = (1 + \sqrt{2 \kappa})^2 b^2 = R^2 
\]
and, for $\kappabar \ge 3$,
\[
Q(0) = (2 - \kappabar^{-1}) b^2 < (1+\sqrt{2\kappabar})^2 b^2= R^2,
\]
Therefore, if $\norm{x}^2 \le R^2$ and $\eta = 1 / \betabar$, then $\norm{x^+}^2 \le R^2$ too. It follows from Lemma~\ref{lem:eta=1/beta} that $\norm{x^+}^2 \le R^2$ for any $\eta \in [0, 1/\betabar]$, which completes the proof.\qed
\end{proof}
\smallskip
The previous proposition requires $\norm{x}\geq b$ and a condition number $\kappabar= \betabar/\alphabar \geq 3$. We will see that smaller condition numbers can easily be treated.
To complete our analysis, we just need thus to focus on $x$ with norms smaller than the bound $b$ on $x^*$. 

\begin{proposition}  \label{prop:stability-small-x}
If $\norm{x} \le b$ then $\norm{x^+} \le 3b$.
\end{proposition}
\begin{proof}
We focus again first on the case 
$\eta = 1/\betabar$. Using Theorem \ref{thm:gradient-method-convergence}, we have 
$$\norm{x^+ - x^*}\leq \sqrt{1- \kappa^{-1}}\norm{x - x^*}$$
Using  $\norm{x} \le b$ and $\norm{x^*}\le b$, which implies $\norm{x - x^*}\le 2b$ we have then
$$
\norm{x^+} \le \norm{x^+ - x^*} + \norm{x^*} \le \norm{x - x^*} + \norm{x^*} \leq 3b.
$$
The result for other values of $\eta \in [0, 1/\betabar]$ follows from Lemma~\ref{lem:eta=1/beta}.\qed
\end{proof}
\smallskip

\modif{We combine} the results above in the following Proposition. 
\begin{proposition}\label{prop:stab_GD}
Let \[
R \defeq \begin{cases} (1 + \sqrt{2 \kappabar})b & \text{ if } \kappabar \geq 3, \\ (1 + \sqrt{6})b& \text{ otherwise.} \end{cases}
\]
If $\norm{x}\leq R$ then $\norm{x^+}\leq R$.
\end{proposition}
\begin{proof}
Suppose first that $\kappabar \geq 3$. If $\norm{x}\geq b$, the result directly follows from Proposition \ref{prop:stability-large-kappa}. If $\norm{x} < b$, Proposition \ref{prop:stability-small-x} and $\kappabar \geq 3$ imply that
$$
\norm{x^+} \le 3b \leq (1+\sqrt{6}) b\leq \modif{(1+\sqrt{2\kappabar})} b.
$$
Smaller condition numbers $\kappabar=\betabar/\alphabar$ can be artificially increased to $\kappabar'=3$ e.g. by considering a lower $\alphabar' = \betabar'/3 < \alphabar$. An $\alphabar$-strongly convex function is indeed always $\alphabar'$-strongly convex for any positive $\alphabar'<\alphabar$. The first part of the result can then be applied with $R= (1 + \sqrt{2 \kappabar'})b = (1 + \sqrt{6})b $. \qed
\end{proof}

\subsection{Proof of Theorem \ref{thm:stability}}

\begin{proof}
The proof essentially follows by specializing the result of Proposition \ref{prop:stab_GD} to the setting of~\eqref{eq:dynamic-algo} and \eqref{eq:dynamic-objective}. By Theorem~\ref{thm:x*localization} we know that
\[
\arg\min_{\bx} F_\rho^k(\bx) \in B(\zeros_d, 1 + \sqrt{\kappa})^n, \quad \text{ for all } k \ge 0,
\]
where $\kappa = \beta / \alpha$, and so we have $\norm{\arg\min_{\bx} F_\rho^k(\bx) } \le b:= \sqrt{n}(1 + \sqrt{\kappa}) $. In addition, it follows from Proposition \ref{prop:f_properties} that the condition number of $F_\rho^k$ is $\kappa_\rho = (\beta + \rho\lambda_n)/\alpha$ for all $k \ge 0$. The result follows then 
from Proposition \ref{prop:stab_GD}. \qed

\end{proof}

\section{Conclusion}

We have shown that DGD is stable: even if all agents change their objectives at every iteration, if the local per-agent objectives are constrained to be smooth, strongly convex, and their minimizer has bounded norm, then the iterates of DGD cannot be made to grow unbounded. 

This paper contributes an initial exploration into the behavior of DGD, and consensus optimization methods in general, in the setting of open multi-agent systems. There are many interesting directions for future work, such as understanding the behavior of more advanced algorithms (e.g., those using gradient tracking), characterizing a stability region under stronger or different assumptions (e.g., if the frequency of function changes is constrained to be slower than the DGD update iterations, or alternative assumptions about 
local objective functions), and developing robust decentralized methods with tighter stability regions than DGD.


\newpage

\appendix

\modif{In this Appendix we prove Theorem~\ref{thm:bound_xa-xb}.} Let $f^- \defeq \sum_{i=1}^{n-1} f_i$ denote the sum of the common terms on the right-hand sides of the definitions of $\xa$ and $\xb$ in~\eqref{eq:xab}. Also let $x^- \defeq \arg\min_x f^-(x)$ denote the minimizer of $f^-$. One approach to bound $\norm{\xa - \xb}$ is to bound the distance between the minimizers of $f^-$ and any function of the form $f^- + f_n$ where $f_n \in \mathcal{F}_{\alpha, \beta}$; i.e., to bound $\norm{x^* - x^-}$ where $x^* = \arg\min_x \big(f^-(x) + f_n(x)\big)$. \modif{Observe indeed that both $\xa$ and $\xb$ are particular cases of such $x^*$. Hence a bound on all $\norm{x^* - x^-}$ together with the triangular inequality}
\begin{equation} \label{eq:triangle}
\modif{\norm{\xa - \xb} \le \norm{\xa - x^-} + \norm{\xb - x^-}} 
\end{equation}
\modif{gives us a bound on $\norm{\xa - \xb}$.} To this end, we have the following.

\smallskip

\begin{theorem} \label{thm:sensitivity}
Let $f = f^-  + f_n$, where $f^-$ is as defined above and $f_n \in \mathcal{F}_{\alpha, \beta}$. Let $x^* = \arg\min_x f(x)$ and $x^- = \arg\min_x f^-(x)$. Then
\begin{equation} \label{eq:sensitivity}
\norm{x^* - x^-} \le \min\left( 2 + 2 \sqrt{\kappa}, \frac{2 \sqrt{\kappa} + \kappa}{\sqrt{n-1}}, \frac{2\kappa + \kappa^{3/2}}{n} \right).
\end{equation}
\end{theorem}
\smallskip

The proof of Theorem~\ref{thm:bound_xa-xb} follows directly from the application of Theorem~\ref{thm:sensitivity} \modif{to both terms in \eqref{eq:triangle}}. Thus, all that remains is to prove Theorem~\ref{thm:sensitivity}. Before \modif{that,}
we derive the following intermediate result.
\smallskip

\begin{proposition} \label{prop:sensitivity-intermediate}
Let $g_1$ and $g_2$ be two functions such that $g_1$ is $\alpha_1$-strongly convex and $\beta_1$-smooth and $g_2$ is $\alpha_2$-strongly convex and $\beta_2$-smooth. Assume further that the minimizer $x_1^*$ of $g_1$ lies in $B(0, R_1)$ and the minimizer $x_2^*$ of $g_2$ lies in $B(0,R_2)$. Then $x^* = \arg\min_x \big( g_1(x) + g_2(x) \big)$ satisfies
\[
\norm{x^* - x_1^*} \le \frac{\beta_2 (R_1 + R_2)}{\alpha_1 + \alpha_2}.
\]
\end{proposition}
\smallskip

\begin{proof}
It follows from the strong convexity and smoothness of $g_1$ and $g_2$ that
\begin{align*}
(\nabla g_1(x^*) - \nabla &g_1(x_1^*))^\T (x^* - x_1^*) \ge \frac{\alpha_1 \beta_1}{\alpha_1 + \beta_1} \norm{x^* - x_1^*}^2 \\ &+ \frac{1}{\alpha_1 + \beta_1} \norm{\nabla g_1(x^*) - \nabla g_1(x_1^*)}^2 \\
(\nabla g_2(x^*) - \nabla &g_2(x_1^*))^\T (x^* - x_1^*) \ge \frac{\alpha_2 \beta_2}{\alpha_2 + \beta_2} \norm{x^* - x_1^*}^2 \\ &+ \frac{1}{\alpha_2 + \beta_2} \norm{\nabla g_2(x^*) - \nabla g_2(x_1^*)}^2.
\end{align*}
Summing these inequalities while taking into account that $\nabla g_1(x_1^*) = 0$ and $\nabla (g_1 + g_2)(x^*) = 0$, we obtain
\begin{align*}
- \nabla g_2(x_1^*)^\T (x^* - x_1^*) &\ge \frac{1}{\alpha_1 + \beta_2}\norm{\nabla g_1(x^*) - \nabla g_1(x_1^*)}^2 \\ &\quad + \frac{1}{\alpha_2 + \beta_2} \norm{\nabla g_2(x^*) - \nabla g_2(x_1^*)}^2 \\
&\quad + \norm{x^* - x_1^*}^2 \left(\frac{\alpha_1 \beta_1}{\alpha_1 + \beta_1} + \frac{\alpha_2 \beta_2}{\alpha_2 + \beta_2}\right).
\end{align*}
Recall that for an $\alpha$-strongly convex function $f$, for any $x$ and $y$ we have $\norm{\nabla f(x) - \nabla f(y)} \ge \alpha \norm{x - y}$. Therefore,
\begin{align}
&- \nabla g_2(x_1^*)^\T (x^* - x_1^*) \nonumber \\
&\ge \norm{x^* - x_1^*}^2 \left(\frac{\alpha_1 \beta_1}{\alpha_1 + \beta_1} + \frac{\alpha_2 \beta_2}{\alpha_2 + \beta_2} + \frac{\alpha_1^2}{(\alpha_1 + \beta_1)} + \frac{\alpha_2^2}{\alpha_2 + \beta_2} \right) \nonumber \\
&= \norm{x^* - x_1^*}^2 (\alpha_1 + \alpha_2). \label{eq:sensitivity-intermediate1}
\end{align}
By the Cauchy-Schwartz inequality, $-\nabla g_2(x_1^*)^\T (x^* - x_1^*) \le \norm{\nabla g_2(x_1^*)} \norm{x^* - x_1^*}$. In addition, using the smoothness of $g_2$ along with the assumptions that $x_1^* \in B(0,R_1)$ and $x_2^* \in B(0,R_2)$, we obtain that
\begin{align*}
\norm{\nabla g_2(x_1^*)} &= \norm{\nabla g_2(x_1^*) - \nabla g_2(x_2^*)} \\
&\le \beta_2 \norm{x_1^* - x_2^*} \\
&\le \beta_2 (\norm{x_1^*} + \norm{x_2^*}) \\
&\le \beta_2 (R_1 + R_2).
\end{align*}
Combining these in \eqref{eq:sensitivity-intermediate1} gives
\[
\beta_2 (R_1 + R_2) \norm{x^* - x_1^*} \ge (\alpha_1 + \alpha_2) \norm{x^* - x_1^*}^2,
\]
and rearranging terms completes the proof.
\end{proof}
\medskip

\begin{proof}[Proof of Theorem~\ref{thm:sensitivity}]
First observe that $f^-$ is $(n-1)\alpha$-strongly convex and $(n-1)\beta$-smooth by similar reasoning as for $f$ in Proposition~\ref{prop:f_properties}. By Theorem~\ref{thm:x*localization} we have that both $x^*$ and $x^-$ lie in $B(0, 1 + \sqrt{\kappa})$, and thus
\begin{equation} \label{eq:sensitivity-bound1}
\norm{x^* - x^-} \le \norm{x^*} + \norm{x^-} \le 2 + 2\sqrt{\kappa}.
\end{equation}
This bound scales as $\order{\sqrt{\kappa}}$ but does not involve $n$.

To obtain a bound involving $n$, we leverage properties of the functions $f^-$ and $f_n$. Recall that $x_n^*$ denotes the minimizer of $f_n$. Since $f_n$ is $\beta$-smooth,
\begin{align} \label{eq:fa-smooth}
f(x^-) = f^-(x^-) + f_n(x^-) &\le f^-(x^-) + \frac{\beta}{2}\norm{x^- - x_n^* }^2.
\end{align}
Also, since $f^-$ is $(n-1)\alpha$-strongly convex and $\min_x f_n(x) = 0$ by assumption, it follows that for any $x \in \R^d$,
\begin{align} 
f(x) &= f^-(x) + f_n(x) \\
&\ge f^-(x) \\
&\ge f^-(x^-) + \frac{(n-1)\alpha}{2} \norm{x - x^-}^2. \label{eq:f-strongly_convex}
\end{align}
Combining equations~\eqref{eq:fa-smooth} and~\eqref{eq:f-strongly_convex} leads to
\begin{align*}
f(x) - f(x^-) &\ge f^-(x^-) + \frac{(n-1)\alpha}{2}\norm{x - x^-}^2 \\
&\quad - f^-(x^-) - \frac{\beta}{2} \norm{x^- - x_n^*}^2 \\
&= \frac{(n-1)\alpha}{2} \norm{x - x^-}^2 - \frac{\beta}{2} \norm{x^- - x_n^*}^2.
\end{align*}
By \assumptionref{ass:x_i*}, we know that $\norm{x_n^*} \le 1$, and applying Theorem~\ref{thm:x*localization} to $f^-$ gives $\norm{x^-} \le 1 + \sqrt{\kappa}$. Therefore $\norm{x^- - x_n^*} \le 2 + \sqrt{\kappa}$, and so
\begin{align*}
f(x) - f(x^-) &\ge \frac{(n-1)\alpha}{2} \norm{x - x^-}^2 - \frac{\beta}{2} (2 + \sqrt{\kappa})^2 \\
&= \frac{(n-1)\alpha}{2} \left( \norm{x - x^-}^2 - \frac{\kappa (2 + \sqrt{\kappa})^2}{n-1} \right).
\end{align*}
A point $x$ cannot be the minimizer of $f$ if there is another point $x^-$ for which $f(x) - f(x^-) > 0$, or equivalently, if
\[
\norm{x - x^-}^2 > \frac{\kappa (2 + \sqrt{\kappa})^2}{n-1},
\]
and therefore
\begin{equation} \label{eq:sensitivity-bound2}
\norm{x^* - x^-} \le \frac{\kappa + 2 \sqrt{\kappa}}{\sqrt{n-1}}.
\end{equation}

Finally, applying Proposition~\ref{prop:sensitivity-intermediate} with $g_1 = f^-$ (hence $\alpha_1 = (n-1)\alpha$ and $R_1 = 1 + \sqrt{\kappa}$) and  $g_2 = f_n$ (hence $\alpha_2 = \alpha$, $\beta_2 = \beta$, and $R_2 = 1$), we also obtain that
\begin{equation} 
\norm{x^* - x^-} \le \frac{\beta (2 + \sqrt{\kappa})}{n \alpha} = \frac{2 \kappa + \kappa^{3/2}}{n}.
\end{equation}
Combining this with \eqref{eq:sensitivity-bound1} and \eqref{eq:sensitivity-bound2} gives the desired result.
\end{proof}

\end{document}